\pdfoutput=1
\documentclass[10pt]{article}

\usepackage{amsmath,graphicx,hyperref}
\usepackage{amsmath, amssymb, epsfig, graphicx, bm}
\usepackage{cases}

\usepackage[left=1in,top=1in,right=1in,bottom=1in,letterpaper]{geometry}

\usepackage{float}
\usepackage{subfigure}
\usepackage{graphicx}
\usepackage{color}
\usepackage{xcolor} 
\usepackage[table]{xcolor}
\usepackage{makecell}
\usepackage{epstopdf}
\usepackage{amsthm}
\usepackage{booktabs}
\usepackage{pgfplots}
\usepackage{underscore}
\usepackage{amsfonts}
\usepackage{pgfplots}
\usepackage{pgfplotstable}
\usepackage{authblk}
\pgfplotsset{compat=1.17} 
\usepackage{cite, url}
\usepackage[all,cmtip]{xy}
\usepackage{color,hyperref}
\usepackage{algorithm}
\usepackage{algpseudocode}
\algtext*{EndFor}
\algtext*{EndIf}
\algtext*{EndFunction}
\usepackage{flushend}
\usepackage{graphicx}
\usepackage{textcomp}
\usepackage{xcolor}
\usepackage{array}
\usepackage{pifont}     
\usepackage{multirow}

\newcommand{\xmark}{\textcolor{red}{\ding{55}}}

\newcommand{\argmin}{\operatorname*{argmin}}

\newtheorem{definition}{Definition}
\newtheorem{assumption}{Assumption}
\newtheorem{lemma}{Lemma}
\newtheorem{claim}{Claim}
\newtheorem{remark}{Remark}
\newtheorem{theorem}{Theorem}
\newtheorem{corollary}{Corollary}

\newcommand{\PowerSGD}{{\texttt{\textbf{PowerSGD}}}}
\newcommand{\PowerSGDP}{{\texttt{\textbf{PowerSGD+}}}}



\title{From PowerSGD to PowerSGD+: Low-Rank Gradient Compression for Distributed Optimization with Convergence Guarantees}

\author[1]{Shengping Xie\thanks{\texttt{2300010702@stu.pku.edu.cn}}}
\author[1]{Chuyan Chen\thanks{\texttt{chuyanchen@stu.pku.edu.cn}}}
\author[2]{Kun Yuan\thanks{Corresponding author: \texttt{kunyuan@pku.edu.cn}}}

\affil[1]{School of Mathematical Sciences, Peking University}
\affil[2]{Center for Machine Learning Research, Peking University}

\date{}

\begin{document}

\maketitle
\begin{abstract}
Low-rank gradient compression methods, such as \PowerSGD, have gained attention in communication-efficient distributed optimization. However, the convergence guarantees of \PowerSGD\ remain unclear, particularly in stochastic settings. In this paper, we show that \PowerSGD\ does not always converge to the optimal solution and provide a clear counterexample to support this finding. To address this, we introduce \PowerSGDP, which periodically updates the projection subspace via singular value decomposition, ensuring that it remains aligned with the optimal subspace. We prove that \PowerSGDP\ converges under standard assumptions and validate its effectiveness through empirical evaluation on large language model tasks.
\end{abstract}
\vspace{-4pt}

\section{Introduction}
Distributed optimization has emerged as a fundamental paradigm for large-scale signal processing and machine learning. Matrix variables naturally arise in such applications. For example, the weights in each layer of a deep neural network are typically matrices. This paper considers a scenario in which $N$ computing nodes collaborate to solve the distributed optimization problem involving a matrix variable $\bm{X}\in \mathbb{R}^{m \times n} (m \ge n)$:
\vspace{-0.5em}
\begin{align}\label{prob-general}
\min_{\bm{X}\in \mathbb{R}^{m \times n}}\ &f(\bm{X}) := \frac{1}{N}\sum_{i=1}^N f_i(\bm{X}), \\
\mbox{where}\ \ \  &f_i(\bm{X}):={\mathbb{E}_{\xi^{(i)} \sim \mathcal{D}_i}}[F(\bm{X}; \xi^{(i)})]. \nonumber
\end{align}
 When $N=1$, Problem~\eqref{prob-general} reduces to a standard stochastic optimization problem with a vector variable. Here, $\xi^{(i)}$ is a random sample from the local distribution $\mathcal{D}_i$, and $f_i: \mathbb{R}^{m\times n} \rightarrow \mathbb{R}$ is the local loss function. Since each node $i$ can only access its own data to compute local gradient $\nabla F(\bm{X};\xi^{(i)})$, nodes must exchange gradients to obtain a global solution.

Centralized stochastic gradient descent (SGD) is a common approach to Problem~\eqref{prob-general}, where local gradients are globally averaged via either a central server \cite{li2014scaling} or ring All-Reduce primitives \cite{patarasuk2009bandwidth}. However, transmitting full gradients introduces substantial communication overhead, resulting in a key bottleneck in distributed learning \cite{seide20141, chilimbi2014project}. To mitigate this, various compression strategies have been developed. Quantization maps high-precision values to low-bit representations, with notable methods including 1-bit SGD \cite{seide20141}, QSGD \cite{alistarh2017qsgd}, and TernGrad \cite{wen2017terngrad}. Sparsification reduces transmission by discarding less informative entries, with popular examples such as random sparsification \cite{konevcny2016federated, wangni2018gradient} and top-$K$ sparsification \cite{stich2018sparsified}.

This paper studies low-rank communication compression, which projects transmitted gradients onto a low-rank subspace to exploit their inherent structure and reduce communication costs. In deep neural network training, the success of the LoRA fine-tuning method~\cite{hulora} demonstrates the effectiveness of low-rank approximations. Recent studies \cite{zhao2024galore, han2024sltrain, chen2024enhancing}  show that both weights and gradients in neural network pre-training exhibit low-rank structures, motivating a re-examination of low-rank compression in distributed learning.

\vspace{-0.2em}
\PowerSGD\ \cite{vogels2019powersgd} is a prominent low-rank gradient compression method in distributed learning. It employs power iteration to align low-rank projection matrices with the optimal rank-$r$ gradient approximation, replacing costly singular value decomposition (SVD) with QR decomposition and achieving strong empirical performance. Fully integrated into PyTorch, \PowerSGD\ supports efficient All-Reduce aggregation with error feedback, and has been deployed in large-scale systems such as OpenAI DALL-E \cite{ramesh2021zero} and AngelPTM \cite{nie2023angel}. Despite its practical success, its convergence guarantees under standard assumptions remain theoretically unproven.

\vspace{1mm}
\noindent \textbf{Contribution.} 
This paper targets to address the convergence guarantees issue in \PowerSGD\ . Surprisingly, we demonstrate that \PowerSGD\ does not always converge to the optimal solution and provide a clear counterexample to support this finding. To resolve this issue, we introduce \PowerSGDP, which incorporates a safeguard mechanism that periodically updates the projection subspace via SVD, ensuring the gradient approximation remains aligned with the optimal subspace. With this safeguard, we theoretically prove that \PowerSGDP\ \textbf{always} converges under standard assumptions. Our empirical evaluation on large language model tasks validates our results. 
\section{Preliminary}
\textbf{Error Feedback.}\label{error-feedback} 
Error feedback (EF) mitigates the adverse effects of compression by incorporating past residuals into subsequent gradient updates, thereby preserving more informative signals \cite{seide20141} \cite{karimireddy2019error}. EF21-MSGD \cite{richtarik2021ef21} extends this concept by maintaining a local gradient tracker for each node, reducing the impact of data heterogeneity and improving convergence rates. Building on this theoretical foundation, NEOLITHIC \cite{huang2022lower} establish lower bounds for distributed learning under communication compression.

\vspace{1mm}
\noindent \textbf{Contractive Compressor.} 
Contractive compressors are commonly employed in communication-efficient distributed optimization. Their key characteristic is that the compression error decreases in expectation as the underlying variable approaches zero. The formal definition of a contractive compressor is provided below:

\begin{definition}{(Contractive Compressor)}.
\label{contractive compressor}
    A compressor $\mathcal{C}$ is {contractive} if there exists a constant $\delta \in (0,1]$ such that for any matrix variable $\bm{\Delta}$, \begin{align}\label{equation2}{\mathbb{E}_{\mathcal{C}}}(\|\mathcal{C}(\bm{\Delta})-\bm{\Delta}\|_F^2) \leq (1-\delta)\|\bm{\Delta}\|_F^2,\end{align}
    where the expectation is taken for the random compressor $\mathcal{C}$.
\end{definition}
\begin{definition}{(Low-rank Compressor)}.
\label{low-rank compressor}
   A {low-rank compressor} $\mathcal{C}$ for a matrix $\bm{W} \in \mathbb{R}^{m \times n}(m \geq n)$ is defined as 
\begin{align}\label{equation3}\mathcal{C} (\bm{W})=\bm{P}\bm{P}^\top \bm{W},
\end{align}
where $\bm{P}\in \mathbb{R}^{m \times r}$ is a semi-orthogonal matrix (i.e., $\bm{P}^\top \bm{P}=I_r$) and $r \leq n$. For $\mathcal{C}$ in ~\eqref{equation3}, it is straightforward to obtain:
\begin{align}\label{equation4}\|\mathcal{C}(\bm{W})-\bm{W}\|_F^2+\|\mathcal{C}(\bm{W})\|_F^2=\|\bm{W}\|_F^2. 
\end{align}
\end{definition}
\begin{definition}{(Optimal rank-$r$ Compressor)}.
\label{svd compressor}
   An {optimal rank-$r$ compressor} is a low-rank compressor as defined in \eqref{equation3}, with the optimal projection matrix $\tilde{\bm{P}}$ that minimizes the compression error:
\begin{align}\label{equation5}
\tilde{\bm{P}} = \underset{\bm{P} \in \mathbb{R}^{m \times r}}\argmin \|\bm{P}\bm{P}^\top \bm{W} - \bm{W}\|_F^2. 
\end{align}
\end{definition}
\noindent The solution to problem \eqref{equation5} is provided SVD. Let $\bm{W} = \bm{U} \Sigma \bm{V}^\top$ be the singular value decomposition of $\bm{W}$. The optimal compressor is given by the first $r$ columns of the left singular matrix, $\tilde{\bm{P}} = \bm{U}[:, :r]$, which satisfies
\begin{align}\label{equation6}
\|\tilde{\bm{P}} \tilde{\bm{P}}^\top \bm{W} - \bm{W}\|_F^2 \leq \left( 1 - \frac{r}{n} \right) \|\bm{W}\|_F^2.
\end{align}
Apparently, the optimal rank-$r$ compressor is contractive. 

\vspace{1mm}
\noindent 
\textbf{Assumptions}. We present the basic assumptions used throughout our theoretical analysis, which are standard in the Error Feedback framework \cite{stich2019error}\cite{fatkhullin2023momentum}\cite{sun2019distributed}.

\begin{assumption}[Smoothness and lower boundedness]\label{assumption}
The global loss function $f(\bm{X})$ is lower-bounded and $L$-smooth.
\end{assumption}
\vspace{-0.5em}
\begin{assumption}[Stochastic gradient]\label{asp:SGO}
For each node $i$, the local stochastic gradient $\nabla F(\bm{X}; \xi^{(i)})$ is unbiased and has bounded variance, i.e.,
$\mathbb{E}_{\xi^{(i)} \sim \mathcal{D}_i}[\nabla F(\bm{X}; \xi^{(i)}] = \nabla f_i(\bm{X})$,
$\mathbb{E}_{{\xi}^{(i)} \sim \mathcal{D}_i}\big[ \|\nabla F(\bm{X}; \xi^{(i)}) - \nabla f_i(\bm{X})\|_F^2 \big] \le \sigma^2$ for any $\bm X$.
\end{assumption}
\begin{assumption}[Uniform gradient bound]\label{ass:uniform}
The full gradient is uniformly bounded, i.e., $\|\nabla f(\bm{X})\|_F^2 \leq \omega^2$.  
This implies that the second moment is also bounded, specifically, 
$
\mathbb{E} \left[ \Big\Vert\frac{1}{N}\sum_{i=1}^N\nabla F(\bm{X},\xi^{(i)} )\Big\Vert_F^2 \right] \leq G^2 \equiv \sigma^2 + \omega^2.
$
\end{assumption}

\noindent \textbf{QR decomposition.} In this paper, we use the function $\mathrm{QR}(\cdot)$ 
to denote matrix orthogonalization by economic QR decomposition: for $\bm{A}\in\mathbb{R}^{n\times r}$ ($n\ge r$), 
it orthogonalizes $\bm{A}=\bm{Q}\bm{R}$ and returns only 
$\bm{Q}\in\mathbb{R}^{n\times r}$ with $\bm{Q}^\top\bm{Q}=I_r$, while 
$\bm{R}$ is omitted.

\section{PowerSGD and its non-convergence}
\def\coloredtitle{\colorbox{red!20}{\PowerSGD}/\colorbox{green!20}{\PowerSGDP}}

\begin{algorithm}[t!]
	\caption{Rank-$r$ SSP/SVD compressor}
	\label{alg:powersgd-compression}
	\begin{algorithmic}[0]
		\State \textbf{Notation:} Set current update matrices $\bm{\Delta}^{(1)},\dots,\bm{\Delta}^{(N)} \in \mathbb{R}^{m \times n}$, previous auxiliary basis $\bm{Q} \in \mathbb{R}^{n \times r}$.
		\Function{SSP-Comp}{$\bm{Q}$, $\bm{\Delta}^{(1)},\dots,\bm{\Delta}^{(N)} $}:
		\State \textbf{on each node $i$}
		\State $\bm{P}^{(i)} \gets \bm{\Delta}^{(i)}\bm{Q}$
		\State $\bm{P} \gets \frac{1}{N}(\bm{P}^{(1)}+\dots+\bm{P}^{(N)})$ \Comment{All-Reduce Mean}
		\State $\tilde{\bm{P}} \gets \text{QR}(\bm{P})$
		\Comment{Power iteration}
		\State $\bm{Q}^{(i)} \gets (\bm{\Delta}^{(i)})^\top \tilde{\bm{P}}$
		\State $\bm{Q} \gets \frac{1}{N}(\bm{Q}^{(1)}+\dots+\bm{Q}^{(N)})$ \Comment{All-Reduce Mean}
		\State \Return $(\bm{Q},\tilde{\bm{P}}(\bm{Q}^{(i)})^\top,\tilde{\bm{P}}\bm{Q}^\top)$ 
		\EndFunction
		\Function{SVD-Comp}{$\bm{\Delta}^{(1)},\dots,\bm{\Delta}^{(N)}$}:
		\State \textbf{on each node $i$}
		\State $\bm{\Delta} \gets \frac{1}{N}(\bm{\Delta}^{(1)}+\dots+\bm{\Delta}^{(N)})$ \Comment{All-Reduce Mean}
		\State $\bm{U},\bm{\Sigma},\bm{V} \gets \text{SVD}(\bm{\Delta})$
		\State $\tilde{\bm{P}} \gets \bm{U}[:,:r]$ \Comment{SVD on full gradient}
		\State $\bm{Q}^{(i)} \gets (\bm{\Delta}^{(i)})^\top \tilde{\bm{P}}$
		\State $\bm{Q} \gets \frac{1}{N}(\bm{Q}^{(1)}+\dots+\bm{Q}^{(N)})$ \Comment{All-Reduce Mean}
		\State \Return $(\bm{Q},\tilde{\bm{P}}(\bm{Q}^{(i)})^\top,\tilde{\bm{P}}\bm{Q}^\top)$ 
		\EndFunction
	\end{algorithmic}
\end{algorithm}

As a well-known communication-efficient algorithm,  \PowerSGD~\cite{vogels2019powersgd} reduces communication by projecting the gradients onto a low-rank  subspace. To mitigate compression error, it employs an error-feedback mechanism that carries the previous compression residual forward into subsequent updates, thereby preserving more information. The \PowerSGD\ framework with MSGD optimizer is summarized below.
\begin{subequations}
\begin{align}
	&\ \bm{g}_t^{(i)} \gets \nabla F(\bm{X}_t, \xi_{t}^{(i)}), 
	   \quad \bm{\Delta}_t^{(i)} \gets \bm{g}_t^{(i)} + \bm{e}_t^{(i)}, \label{powerSGD-1}\\
	&\ \bm{Q}_t, \widehat{\bm{\Delta}}_t^{(i)}, \widehat{\bm{\Delta}}_t  
	   \gets \textsc{Comp}(\bm{Q}_{t-1},\{\bm{\Delta}_t^{(i)}\}_{i=1}^N), \label{powerSGD-2} \\
	&\ \bm{m}_t=\mu \bm{m}_{t-1} + \widehat{\bm{\Delta}}_t, 
	   \quad \bm{X}_{t+1} \gets \bm{X}_t - \eta \cdot \bm{m}_t. \label{powerSGD-3}
\end{align}
\end{subequations}
Here, Step~\eqref{powerSGD-1} is executed in parallel on each node $i$, where 
$\bm{\Delta}_t^{(i)}$ denotes the local stochastic gradient corrected by the previous 
compression residual $\bm{e}_t^{(i)}$. In Step~\eqref{powerSGD-2}, the function  
$\mbox{\textsc{Comp}}(\cdot)$ transforms both the local gradients $\bm{\Delta}_t^{(i)}$ and the 
averaged global gradient 
$\bm{\Delta}_t = \tfrac{1}{N}\sum_{i=1}^N \bm{\Delta}_t^{(i)}$ into their low-rank 
approximations,  $\widehat{\bm{\Delta}}_t^{(i)}$ and $\widehat{\bm{\Delta}}_t$ using a specific low-rank compressor $\mathcal{C_{\text{SSP}}}$. The matrix $\bm{Q}_t$ serves as the auxiliary basis for compression. 
Step~\eqref{powerSGD-3} then applies a standard momentum SGD update. Finally, each node updates 
its residual as 
$\bm{e}_{t+1}^{(i)} \gets \bm{\Delta}_t^{(i)} - \widehat{\bm{\Delta}}_t^{(i)}$.

\vspace{1em}
\noindent \textbf{Single-step power iteration.} 
The low rank compressor $\mathcal{C_{\text{SSP}}}$ in \PowerSGD\ relies on single-step power iteration. 
Given the global gradient $\bm{\Delta}_t$ and an initial basis 
$\bm{Q}_{-1} \in \mathbb{R}^{n\times r}$, the standard power iteration runs for 
$k = 0, \ldots, K-1$ as
\begin{align}
    \bm{{\tilde{P}}}_{k} = \mathrm{QR}(\bm{\Delta}_{t} \bm{Q}_{k-1}), \qquad \bm{Q}_k &= \bm{\Delta}_t^\top \bm{\tilde{P}}_k,
\end{align}
With the resulting projection 
matrices, $\widehat{\bm{\Delta}}_t = \bm{P}_K \bm{Q}_K^\top$ serves as a near-optimal 
rank-$r$ approximation of $\bm{\Delta}_t$. To reduce both computation and communication, \PowerSGD\ employs only a 
\emph{single iteration} in compressor $\mathcal{C}(\cdot)$ of step \eqref{powerSGD-2}:

\vspace{-1.5em}
\begin{align}\label{single-power}
\bm{\tilde{P}}_{t} = \mathrm{QR}(\bm{\Delta}_{t} \bm{Q}_{t-1}), \ \bm{Q}_t = \bm{\Delta}_t^\top \bm{\tilde{P}}_t, \ \widehat{\bm{\Delta}}_{t} =\mathcal{C}_{\text{SSP}}(\bm{\Delta}_{t})=\bm{\tilde{P}}_{t} \bm{Q}_t^\top=\bm{\tilde{P}}_{t} \bm{\tilde{P}}_t^\top\bm{\Delta}_{t}.
\end{align}

This single-step scheme is implemented efficiently across $N$ nodes as the 
\textsc{SSP-Comp}$(\cdot)$ function in Algorithm~\hyperref[alg:powersgd-compression]{1}, 
where \textsc{SSP} denotes ``single-step power iteration''. The implementation of \colorbox{red!20}{\PowerSGD} is in Algorithm~\hyperref[powersgdalgo]{2}. 
\begin{figure}[t!]
\centering
\begin{minipage}{0.477\textwidth}
    \raggedleft
    \includegraphics[width=6.45cm, height=6.0cm]{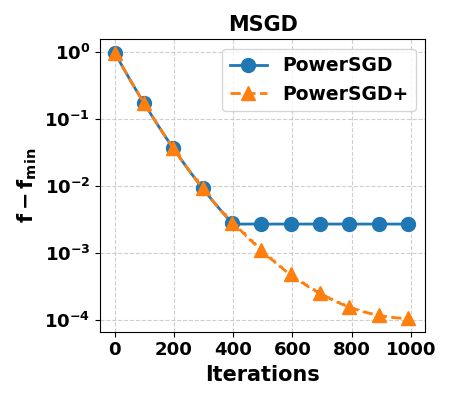}
\end{minipage}
\begin{minipage}{0.477\textwidth}
    \includegraphics[width=6.45cm, height=6.0cm]{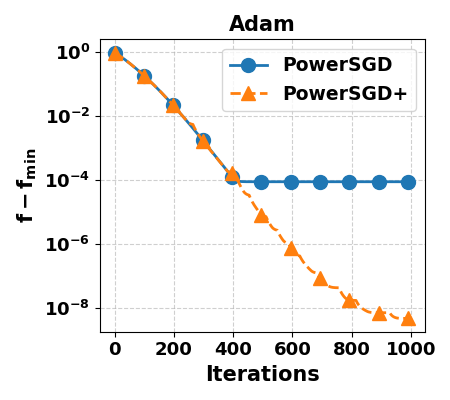}
\end{minipage}
\vspace{-0.6em}
\caption{\small Loss curves of \PowerSGD\ and \PowerSGDP\ using MSGD and Adam.}
\label{figure1}
\end{figure}

\vspace{1em}
\noindent \textbf{Non-convergence of PowerSGD.} With only one power step, \PowerSGD\ acts as an online subspace-tracking method. Compared to standard power iteration, SSP uses $\bm{Q}_{t-1}$ as an estimation of $\bm{Q}_{t}$ in ~\eqref{single-power}, which implicitly {\em assumes the leading subspace of the gradient $\bm \Delta_t$ changes slowly}. When this assumption is violated—such as under rapidly varying or highly noisy gradients—the projection matrix may lock onto an incorrect or even orthogonal subspace, preventing convergence to the desired solution. 

To illustrate this insight, we construct an example with the objective function and stochastic gradient in \eqref{prob-general} defined as
\begin{subequations}
\label{counter}
\begin{align}
f_i(\bm{X})&=\psi(\bm{D}^\top\bm{X}\bm{D}), \label{counter-1}\\
\nabla F(\bm{X};\xi^{(i)})&=\nabla f(\bm{X})+\xi^{(i)}\big(\sigma\bm{B}-\nabla f(\bm{X})\big). \label{counter-2}
\end{align} 
\end{subequations}
where $\bm{X} \in \mathbb{R}^{2 \times 2}$ is the optimization variable, $\bm{D}=[1;-1]\in \mathbb{R}^2$, $\bm{B}=[1,1;1,1]\in \mathbb{R}^{2\times 2}$ and $\sigma\geq 0$. The function $\psi(\cdot)$ is defined as $\psi(x) = x^2$ if $|x|\le 1$ and $\psi(x) = 2|x|-1$ otherwise. At each iteration, the noise variable $\xi^{(i)}$ is uniformly sampled from $\{-1, 1\}$ independently across all nodes. It is straightforward to verify that problem~\eqref{counter} satisfies Assumptions~\ref{assumption}--\ref{ass:uniform}. Moreover, for large $\sigma$, the stochastic gradient in~\eqref{counter-2} fluctuates significantly due to the randomness of $\xi^{(i)}$. We show that \PowerSGD\ fails to converge for this example:
\begin{theorem}[Non-Convergence of \PowerSGD]
\label{theorem1}
There exist local functions $f_i: \mathbb{R}^{2\times 2}\to \mathbb{R}$ and a stochastic gradient oracle satisfying Assumptions ~\ref{assumption}--~\ref{ass:uniform} such that for any choice of initialization $\bm{Q}=\bm{Q}_{-1}$, learning rate schedule $\eta_t$, and base optimizer (e.g., SGD, MSGD, Adam), there is a constant $\epsilon_0=\epsilon(\bm X_0) > 0$ for which the gradient generated by \PowerSGD\ at any time step $t$ is bounded away from $0$ (See proof in Appendix ~\hyperref[appendixa]{A}):
\vspace{-1em}
$$\mathbb{E}(\|\nabla f(\bm{X}_t)\|_F^2) \geq \epsilon_0.$$
\end{theorem}
\noindent Figure~\ref{figure1} illustrates a more general counterexample through 
numerical experiments. The plots show that \PowerSGD\ fails to converge 
under both MSGD and Adam, leading to stagnation of the training loss.

\begin{table*}[htb!]
\small
\centering
\caption{\small GLUE test results for fine-tuning \texttt{RoBERTa-Base} (higher is better). ``Avg'' is the mean across tasks.}
\label{table:glue}
\footnotesize
\resizebox{\textwidth}{!}{
\begin{tabular}{l|cccccccc|c}
\toprule
\textbf{Algorithm} & \textbf{CoLA} & \textbf{STS-B} & \textbf{MRPC} & \textbf{RTE} & \textbf{SST-2} & \textbf{MNLI} & \textbf{QQP} & \textbf{QNLI} & \textbf{Avg} \\
\midrule
Full precision       & \textbf{62.89} & 90.29 & 92.41 & \textbf{79.78} & \textbf{94.38} & 87.20 & \textbf{91.74} & \textbf{92.42} & \textbf{86.35} \\
\midrule
PowerSGD ($r=4$)    & 60.58 & 90.65 & 91.68 & 78.33 & 94.15 & 87.47 & 91.17 & 92.07 & 85.73 \\
GaLore ($r=4$)       & 62.64 & 89.68 & 91.96 & 77.62 & 94.15 & \textbf{87.50} & 90.53 & 92.07 & 85.70 \\
\textbf{PowerSGD+ ($r=4$)} & 62.85 & \textbf{90.67} & \textbf{92.54} & 79.42 & \textbf{94.38} & 87.08 & 91.14 & 92.13 & 86.28 \\
\bottomrule
\end{tabular}
}
\vspace{-2em}
\end{table*}

\begin{algorithm}[t!]
\label{powersgdalgo}
\caption{\coloredtitle}
\textbf{Require:} stepsize $\eta$; momentum parameter $0 \leq \mu < 1$; restart step $\tau$; compression rank $r$; number of workers $N$.\\
\textbf{Initialize:} initial weight $\bm{X}_0 \in \mathbb{R}^{m \times n}$, reflection matrix $\bm{Q}_{-1} \in \mathbb{R}^{n \times r}$, momentum $\bm{m}_{-1}= 0_{m \times n}$, error $\bm{e}_0^{(i)}=0_{m \times n}$ for worker $i$. All variables are set separately on all nodes;
\begin{algorithmic}[0]
\For{$t = 0, \ldots, T - 1$}
\State \textbf{on each node $i$}\\
\hspace{4.5mm} Update $\bm{g}_t^{(i)}$ and $\bm{\Delta}_t^{(i)}$ as in \eqref{powerSGD-1}. \\

\noindent \hspace{-0.3em}\colorbox{red!20}{\parbox{\dimexpr\linewidth-1\fboxsep}{%
    \State $\bm{Q}_t,\widehat{\bm{\Delta}}_t^{(i)},\widehat{\bm{\Delta}}_t\gets \Call{QR-COMP}{\bm{Q}_{t-1},{\{\bm{\Delta}_t^{(i)}}\}_{i=1}^N}$
    \State \hfill \makebox[0pt][r]{(\PowerSGD )}
}}

\noindent \hspace{-0.3em}\colorbox{green!20}{\parbox{\dimexpr\linewidth-1\fboxsep}{%
    \If{$\text{mod}(t, \tau) \neq 0$}
        \State $\bm{Q}_t,\widehat{\bm{\Delta}}_t^{(i)},\widehat{\bm{\Delta}}_t\gets \Call{QR-COMP}{\bm{Q}_{t-1},\{{\bm{\Delta}_t^{(i)}}\}_{i=1}^N}$
    \EndIf
    \If{$\text{mod}(t, \tau) = 0$}
        \State $\bm{Q}_t,\widehat{\bm{\Delta}}_t^{(i)},\widehat{\bm{\Delta}}_t\gets \Call{SVD-COMP}{\{{\bm{\Delta}_t^{(i)}}\}_{i=1}^N}$
    \EndIf \hfill \makebox[0pt][r]{(\PowerSGD+ )}
}}        
\State $\hspace*{\dimexpr-\algorithmicindent} \bm{e}_{t+1}^{(i)} \gets \bm{\Delta}_t^{(i)} - \widehat{\bm{\Delta}}_t^{(i)}$\\
\hspace{4.5mm} Update $\bm{m}_t$ and $\bm{X}_{t+1}$ as in \eqref{powerSGD-3}.
\EndFor
\State \hspace{-5.3mm}\Return ${\{\bm{X}_t\}}_{t=0}^{T}$
\end{algorithmic}
\end{algorithm}

\section{Safeguard to ensure convergence}
\noindent \textbf{From PowerSGD to PowerSGD+.} \PowerSGD\ relies on a single-step power iteration throughout the optimization process, which makes it susceptible to misalignment between the identified subspace and the true gradient subspace. To address this issue, we introduce \textbf{\PowerSGD+}, a simple variant that periodically resets the projection subspace by recomputing the optimal rank-$r$ approximation of the gradient. This periodic reset acts as a safeguard, ensuring alignment with the evolving gradient and guaranteeing convergence.

As shown in Algorithm~\hyperref[powersgdalgo]{2}, \colorbox{green!20}{\PowerSGD+} periodically replaces compressor $\mathcal{C}_{\text{SSP}}$ with the optimal rank-$r$ compressor $\mathcal{C}_{\text{SVD}}$ in Eq.~\eqref{equation5}. The full implementation of $\mathcal{C}_{\text{SVD}}$ is shown in function \textsc{SVD-Comp}$(\cdot)$ . This involves performing an SVD on the full gradient and transmitting the entire matrix. While full-gradient communication is costly, it occurs only once every $\tau$ iterations, making the communication overhead minimal and less impactful on overall performance.

\vspace{1mm}
\noindent \textbf{Convergence guarantees.} By periodically applying the optimal rank-$r$ compression, which prevents the identified gradient subspace from deviating too much during single-step power iteration, we can establish convergence guarantees.
\allowdisplaybreaks 
\begin{theorem}{(Convergence of \PowerSGD+\ )}
\label{theorem2}Under Assumption ~\ref{assumption}-~\ref{ass:uniform}, Define $\delta=r/n$ and $\Delta F=f(\bm{X}_0)-\underset{}{\text{min}}\ f(\bm{X})$. If we set the learning rate $\eta$ properly, \PowerSGD+\ with MSGD (Algorithm \hyperref[powersgdalgo]{2}) converges as follows (See proof in Appendix ~\hyperref[appendixb]{B})
\begin{align*}
\ \frac{1}{T}\sum_{t=0}^{T - 1} \mathbb{E} \left[ \|\nabla f(\bm{X}_t)\|_F^2 \right]
\leq\ \sqrt{\frac{32L\sigma^2\Delta F}{{NT}}} +\frac{8L\Delta F}{T}+\frac{6\sqrt[3]{274}(L\tau G \Delta F)^{2/3} }{T^{2/3}\delta^{2/3}(1-\mu)^{2/3}}.
\end{align*}
\end{theorem}
\begin{remark}
When $T \to \infty$, \PowerSGD+\ achieves a convergence rate of $\mathcal{O}(1/\sqrt{NT})$ which is \textbf{irrelevant} to the restart period $\tau$. It requires $T = \mathcal{O}(1/(N\epsilon^2))$ iterations to reach a target accuracy, and this achieves linear speedup. 
\end{remark}
\begin{remark}
For variants of \PowerSGD+\ which replace $\mathcal{C}_{\text{SVD}}$ with other contractive compressor of coefficient $\delta$, the same convergence rate is also achieved.
\end{remark}
\noindent \textbf{Connections with other algorithms.} We highlight key distinctions between \PowerSGD+ and other communication-efficient methods. ATOMO~\cite{wang2018atomo} applies SVD to local gradients, yielding inconsistent bases across nodes and preventing all-reduce, while GaLore~\cite{zhao2024galore} updates its projection matrix via Lazy-SVD~\cite{chen2025greedy}, refreshing only once every $\tau$ steps. In contrast, \PowerSGD~\cite{vogels2019powersgd} relies solely on power iteration without corrective SVD. Table~\hyperref[tab:communication_comparison]{1} compares communication efficiency: relative to GaLore, \PowerSGD+ incurs an extra $mr$ cost, but by aligning the projection matrix at every step, it enables larger $\tau$, making this overhead negligible. When smaller $\tau$ is required, the full SVD can be replaced with lighter contractive compressors, such as Approximation Top-$r$~\cite{chen2025greedy} or random projection~\cite{he2024subspace}, allowing \PowerSGD+ to match \PowerSGD’s communication complexity while ensuring convergence guarantees.

\begin{table}[t]
  \centering
  \caption{\small Communication comparison of existing low-rank algorithms for compressing an $m \times n$ matrix.}
  \label{tab:communication_comparison}
  \setlength{\tabcolsep}{4pt} 
  \small
  \begin{tabular}{lccc}
    \toprule
    \textbf{Algorithm} & \textbf{All-Reduce} & \textbf{Error-Feedback} & \textbf{\makecell{Communication\\per Iteration}} \\
    \midrule
    ATOMO\cite{wang2018atomo}      & \textcolor{blue}{\xmark}   & \textcolor{green}{\checkmark} & $mr + nr$ \\
    PowerSGD \cite{vogels2019powersgd}   & \textcolor{green}{\checkmark} & \textcolor{green}{\checkmark} & $mr + nr$ \\
    GaLore \cite{zhao2024galore}     & \textcolor{green}{\checkmark} & \textcolor{blue}{\xmark}   & $nr + \frac{mn}{\tau}$ \\
    \rowcolor{blue!10} 
    PowerSGD+         & \textcolor{green}{\checkmark} & \textcolor{green}{\checkmark} & $nr + mr + \frac{mn}{\tau}$ \\
    \rowcolor{blue!30}
    \makecell{PowerSGD+(Appro-\\ximation Top-r)\cite{chen2025greedy}} & \textcolor{green}{\checkmark} & \textcolor{green}{\checkmark} & $nr + mr$ \\
    \bottomrule
  \end{tabular}
\end{table}

\section{Experiment}
\label{section5}
We evaluate \PowerSGD+\ by pre-training and fine-tuning using four NVIDIA RTX~4090 (24GB) GPUs. To ensure fair comparison, we run the Galore with full optimizer state \cite{zhao2024galore} and \PowerSGD\ \cite{vogels2019powersgd} with same compression rank and compared the performance of various compression methods against full-parameter training with the standard AdamW optimizer. \PowerSGD\ and \PowerSGD+\ are applied with error feedback. 
\vspace{1em}

\noindent \textbf{Fine-tuning with PowerSGD+ .}
We fine-tune the pretrained \texttt{RoBERTa-Base} model on the General Language Understanding Evaluation (GLUE) benchmark. We conducted a grid search to select the optimal hyperparameter configuration tailored for \PowerSGD\ with rank 4 and transfer the hyperparameters to all methods. We fine-tune for 10 epochs per task and report test metrics in Table \ref{table:glue}. Detailed hyperparameters are listed in Table ~\ref{tab:finetuning_hparams}.
\vspace{1em}

\noindent \textbf{Pre-training with PowerSGD+ .} We pre-train LLaMA models on the C4 corpus. Following \cite{zhao2024galore}, the 60M and 130M models are trained for 10k and 20k iterations. We report the best validation perplexity in Table~\ref{table:c4} after tuning the basic hyperparameters and peak learning rate over $\{1\times10^{-3},\,2\times10^{-3},\,4\times10^{-3}\}$. Detailed hyperparameters are listed in Table ~\ref{c4-hyperparameters}.

\begin{table}[t]
\centering
\caption{\small Validation perplexity (PPL) for pre-training LLaMA on C4 (lower is better). GaLore diverged at ranks $r\in\{4,8\}$}.
\label{table:c4}
\begin{tabular}{l|cc|cc}
\toprule
\textbf{Algorithm} & \multicolumn{2}{c|}{\textbf{LLaMA-60M}} & \multicolumn{2}{c}{\textbf{LLaMA-130M}} \\
\midrule
Full rank & \multicolumn{2}{c|}{29.88} & \multicolumn{2}{c}{23.06} \\
\midrule
\textbf{Rank} & $r=4$ & $r=8$ & $r=4$ & $r=8$ \\
\midrule
GaLore    & 94.96 & 92.91 & 92.93 & 70.55 \\
PowerSGD  & 36.11 & \textbf{32.57} & 27.60 & 26.24 \\
PowerSGD+ & \textbf{35.46} & 32.63 & \textbf{27.36} & \textbf{26.11} \\
\bottomrule
\end{tabular}
\end{table}
\vspace{1em}

\noindent \textbf{Analysis.} In fine-tuning tasks, \PowerSGD+\ attains near-parity with the full-rank baseline, leading the scores among the compressors in pre-training tasks, GaLore fails to converge at low ranks, whereas \PowerSGD\ and \PowerSGD+\ match the  baseline at the same ranks. Moreover, \PowerSGD+\ slightly outperforms \PowerSGD\ across most settings. These results indicate that both subspace recomputation and power-iteration alignment in \PowerSGD+ are helpful, achieving a consistent, competitive performance under identical setting.

\section{Conclusion}
In this work, we show that \PowerSGD\ lacks convergence guarantees by presenting a counterexample. To overcome this, we propose \PowerSGD+\ , which periodically updates the subspace via SVD to stay aligned with the optimum. We provide the first rigorous convergence analysis for \PowerSGD+\ and validate its effectiveness through extensive large-scale experiments.

\section{References}
\bibliographystyle{unsrt}
\bibliography{references}

\appendix
\newpage
\section{Non convergence of PowerSGD}
\label{appendixa}
\begin{theorem}[Non-Convergence of \PowerSGD]
\label{theorem1appendix}
There exist local functions $f_i: \mathbb{R}^{2\times 2}\to \mathbb{R}$ and a stochastic gradient oracle satisfying Assumptions ~\ref{assumption}--~\ref{ass:uniform} such that for any choice of initialization $\bm{Q}=\bm{Q}_{-1}$, learning rate schedule $\eta_t$, and base optimizer (e.g., SGD, MSGD, Adam), there is a constant $\epsilon_0=\epsilon(\bm X_0) > 0$ for which the gradient generated by \PowerSGD\ at any time step $t$ is bounded away from $0$:
\vspace{-1em}
$$\mathbb{E}(\|\nabla f(\bm{X}_t)\|_F^2) \geq \epsilon_0.$$
\end{theorem}

\begin{proof}
Let $\bm{X} = \begin{pmatrix}
    x_{11} & x_{12} \\
    x_{21} & x_{22}
\end{pmatrix}$, 
$\psi(x) = 
\begin{cases} 
x^2 &  |x| \leq 1 \\
2|x| - 1 &  |x| \geq 1.
\end{cases}$\\
We set $f_i(\bm{X}) = \psi(x_{11} - x_{12} - x_{21} + x_{22})$ for all nodes, with initialization at $\bm{X}_0 = \begin{pmatrix}a & b \\ c & d\end{pmatrix}$.\\
Define $\bm{A} = \begin{pmatrix}1 & -1 \\ -1 & 1\end{pmatrix}$, $\bm{B} = \begin{pmatrix}1 & 1 \\ 1 & 1\end{pmatrix}$, $\bm{C} = \begin{pmatrix}\frac{\sqrt{2}}{2} \\ \frac{\sqrt{2}}{2}\end{pmatrix}$, and set $\displaystyle \epsilon_0=\frac{[\psi'(a-b-c+d)]^2}{2^{N-2}}$.\\
Let $\bm{g}_{t}^{(i)} \equiv \nabla f(\bm{X}_t) + \xi_t^{(i)} \cdot (\sigma \bm{B} - \nabla f(\bm{X}_t))$ with arbitrary $\sigma \geq 0$, where $\xi_t^{(1)},\dots,\xi_t^{(N)}$ are independent random variables with $P(\xi_t^{(i)}=1) = P(\xi_t^{(i)}=-1) = \frac{1}{2}$.

\noindent Then at any point $\bm{X}_t$, $\nabla f(\bm{X}_t) = \psi'(x_{11} + x_{12} - x_{21} - x_{22})\bm{A} \in \text{span}(\bm{A})$.

\noindent Since $\psi'(x) \leq 2$ for any $x \in \mathbb{R}$, it is clear that Assumptions ~\ref{assumption}--~\ref{ass:uniform} are satisfied.

\vspace{1em}
\noindent Now we prove the non-convergence result by induction.

\begin{claim}
\label{claim1}
If the stochastic gradient at $t=0$ satisfies $\xi_0^{(1)} = \dots = \xi_0^{(N)} = 1$, then for any initialization of $\bm{Q}_{-1}$ and any stochastic gradients at $t \geq 1$, we have for all $t \geq 0$:
\[
\bm{m}_t \in \text{span}(\bm{B}), \quad \bm{e}_{t+1}^{(i)} \in \text{span}(\bm{A}), \quad \tilde{\bm{P}}_t = \bm{C}, \quad \bm{Q}_t \in \text{span}(\bm{C}).
\]
\end{claim}

\begin{proof}
Note that:
\[
\bm{C}\bm{C}^\top\bm{B} = \bm{B}, \quad \bm{C}\bm{C}^\top\bm{A} = 0, \quad \bm{B}\bm{C} = \bm{B}^\top\bm{C} = 2\bm{C}, \quad \bm{A}\bm{C} = \bm{A}^\top\bm{C} = 0.
\]

\noindent When $t=0$, since $\bm{m}_{-1} = \bm{e}_{-1}^{(i)} = 0$, we have:
\[
\bm{\Delta}_{0}^{(i)} = \sigma \bm{B} = 
\begin{pmatrix}
\sigma & \sigma \\
\sigma & \sigma
\end{pmatrix}.
\]

\noindent For any initialization $\bm{Q} = \begin{pmatrix}a \\ b\end{pmatrix}$:
\begin{align*}
\bm{P}_0 &= \frac{1}{N} \sum_{i=1}^N \bm{\Delta}_{0}^{(i)} \bm{Q} = \bm{\Delta}_{0}^{(1)} \bm{Q} = \begin{pmatrix}\sigma(a+b) \\ \sigma(a+b)\end{pmatrix}, \quad \text{so } \tilde{\bm{P}}_0 = \bm{C}.
\end{align*}

\noindent Then:
\begin{align*}
\bm{Q}_0 &= \frac{1}{N} \sum_{i=1}^N \bm{\Delta}_{0}^{(i)\top} \tilde{\bm{P}}_0 = \bm{\Delta}_{0}^{(1)\top} \bm{C} = 2 \sigma \bm{C} \in \text{span}(\bm{C}), \\
\mathcal{C}(\bm{\Delta}_{0}^{(i)}) &= \bm{C}\bm{C}^\top \bm{\Delta}_{0}^{(i)} = \bm{\Delta}_{0}^{(i)}, \\
\bm{e}_1^{(i)} &= \bm{\Delta}_{0}^{(i)} - \mathcal{C}(\bm{\Delta}_{0}^{(i)}) = 0 \in \text{span}(\bm{A}), \\
\bm{m}_0 &= \mathcal{C}(\bm{\Delta}_0) = \sigma \bm{B} \in \text{span}(\bm{B}).
\end{align*}

\noindent For $t \geq 1$, let $\bm{g}_{t}^{(i)} = \lambda_t^{(i)} \bm{A} + \beta_t^{(i)} \bm{B}$, and suppose (by induction) that $\bm{e}_t^{(i)} = \gamma_t^{(i)} \bm{A}$, $\bm{m}_{t-1} = \phi_{t-1} \bm{B}$, and $\bm{Q}_{t-1} = q_{t-1} \bm{C}$.

\noindent Then:
\begin{align*}
\bm{\Delta}_t^{(i)} &= \bm{g}_t^{(i)} + \bm{e}_t^{(i)} = (\lambda_t^{(i)} + \gamma_t^{(i)})\bm{A} + \beta_t^{(i)} \bm{B}, \\
\bm{P}_t &= \frac{1}{N} \sum_{i=1}^N \bm{\Delta}_t^{(i)} \bm{Q}_{t-1} = \frac{q_{t-1}}{N} \left[ \left(\sum_{i=1}^N (\lambda_t^{(i)} + \gamma_t^{(i)}) \right)\bm{A} +  \left( \sum_{i=1}^N \beta_t^{(i)} \right) \bm{B} \right] \bm{C} \in \text{span}(\bm{C}), \\
\Rightarrow \tilde{\bm{P}}_t &= \bm{C}.
\end{align*}

\noindent Then:
\begin{align*}
\bm{Q}_t &= \frac{1}{N} \sum_{i=1}^N \bm{\Delta}_t^{(i)\top} \tilde{\bm{P}}_t = \frac{q_{t-1}}{N} \left[ \left(\sum_{i=1}^N (\lambda_t^{(i)} + \gamma_t^{(i)}) \right)\bm{A}^\top + \left( \sum_{i=1}^N \beta_t^{(i)} \right) \bm{B}^\top \right] \bm{C} \in \text{span}(\bm{C}), \\
\mathcal{C}(\bm{\Delta}_t^{(i)}) &= \bm{C}\bm{C}^\top \bm{\Delta}_t^{(i)} = \beta_t^{(i)}  \bm{B}, \\
\bm{e}_{t+1}^{(i)} &= \bm{\Delta}_t^{(i)} - \mathcal{C}(\bm{\Delta}_t^{(i)}) = (\lambda_t^{(i)} + \gamma_t^{(i)}) \bm{A} \in \text{span}(\bm{A}), \\
\bm{m}_t &= \mu \bm{m}_{t-1} + \frac{1}{N} \sum_{i=1}^N \mathcal{C}(\bm{\Delta}_t^{(i)}) = \left( \mu \phi_{t-1} + \frac{1}{N} \sum_{i=1}^N \beta_t^{(i)} \right) \bm{B} \in \text{span}(\bm{B}).
\end{align*}

\noindent Hence, by induction, the claim holds.
\end{proof}

\vspace{1em}
\noindent Returning to Theorem ~\ref{theorem1appendix}, when $\xi_0^{(1)} = \dots = \xi_0^{(N)} = 1$, we always have:
\[
\bm{X}_t - \bm{X}_0 = \sum_{k=0}^{t-1} \eta_k \bm{m}_k \in \text{span}(\bm{B}).
\]

\noindent Set $\bm{X}_t = \begin{pmatrix}
a + \lambda & b+\lambda \\
c+\lambda & d+\lambda
\end{pmatrix}$, then:
\[
\nabla f(\bm{X}_t) \equiv \psi'(a-b-c+d)\bm{A}.\]

\noindent Therefore:
\begin{align*}
\mathbb{E}(\|\nabla f(\bm{X}_t)\|_F^2) 
&\geq \mathbb{E}(\|\nabla f(\bm{X}_t)\|_F^2 \mid \xi_0^{(1)} = \dots = \xi_0^{(N)} = 1) \cdot P(\xi_0^{(1)} = \dots = \xi_0^{(N)} = 1) \\
&= 4[\psi'(a-b-c+d)]^2 \cdot \frac{1}{2^N} \overset{}{=}\epsilon_0 \ . \\
\end{align*}
\end{proof}
\begin{remark}
For Adam or SGD as base optimizers, similarly we can also prove $\bm{X}_t - \bm{X}_0 \in\text{span}(\bm{B})$, as the compressed gradients at each step is in $\text{span}(\bm{B})$.
\end{remark}

\section{Convergence of PowerSGD+}
\label{appendixb}
\textbf{Notations}.
In the following proof, we define 
$$\bm{\Delta}_t=\frac{1}{N}(\bm{\Delta}_{t}^{(1)} + \dots + \bm{\Delta}_{t}^{(N)})\ ,$$
As $\mathcal{C}$ in Algorithm \hyperref[alg:powersgd-compression]{1} is always a linear compressor, we can get:
$$\mathcal{C}(\bm{\Delta}_t)=\frac{1}{N}(\mathcal{C}(\bm{\Delta}_t^{(1)}) + \dots + \mathcal{C}(\bm{\Delta}_t^{(N)})\ .$$
\begin{lemma}
\label{lemma1}
For operator $\mathcal{C}=\mathcal{C}_{\text{SSP}}$ defined in Algorithm \hyperref[powersgdalgo]{2}, for any $\bm{\Delta}_t$, we have:
$$\|\mathcal{C}(\bm{\Delta}_t)\|_F^2 \leq \|\bm{\Delta}_t\|_F^2, \quad \|\mathcal{C}(\bm{\Delta}_t)-\bm{\Delta}_t\|_F^2 \leq \|\bm{\Delta}_t\|_F^2\ .$$
\end{lemma}
\begin{proof}
Followed by Algorithm \hyperref[powersgdalgo]{2}, we have:
\begin{align*}
\|\mathcal{C}(\bm{\Delta}_t)\|_F^2 &= \|\tilde{\bm{P}}_t\bm{Q}_t\|_F^2 \\
&= \left\| \frac{1}{N}\tilde{\bm{P}}_t \left( \bm{Q}_{t}^{(1)} + \dots + \bm{Q}_{t}^{(N)} \right) \right\|_F^2 \\
&= \left\| \frac{1}{N}\tilde{\bm{P}}_t \left( (\tilde{\bm{P}}_t)^\top {\bm{\Delta}}_{t}^{(1)} + \dots + (\tilde{\bm{P}}_t)^\top {\bm{\Delta}}_{t}^{(N)} \right) \right\|_F^2 \\
&= \left\| \tilde{\bm{P}}_t (\tilde{\bm{P}}_t)^\top {\bm{\Delta}}_t \right\|_F^2\ .
\end{align*}

\noindent As $\tilde{\bm{P}}_t$ is an orthogonal matrix with $(\tilde{\bm{P}}_t)^\top \tilde{\bm{P}}_t = \bm{I}_r$, we have:
\begin{align*}
&\|\mathcal{C}(\bm{\Delta}_t)\|_F^2 + \|\mathcal{C}(\bm{\Delta}_t) - \bm{\Delta}_t\|_F^2 \\
&= \left\| \tilde{\bm{P}}_t (\tilde{\bm{P}}_t)^\top {\bm{\Delta}}_t \right\|_F^2 + \left\| \left( \bm{I} - \tilde{\bm{P}}_t (\tilde{\bm{P}}_t)^\top \right) {\bm{\Delta}}_t \right\|_F^2 \\
&= \text{tr}\left( ({\bm{\Delta}}_t)^\top \left[ \tilde{\bm{P}}_t (\tilde{\bm{P}}_t)^\top \right]^2 {\bm{\Delta}}_t \right) + \text{tr}\left( ({\bm{\Delta}}_t)^\top \left[ \bm{I} - \tilde{\bm{P}}_t (\tilde{\bm{P}}_t)^\top \right]^2 {\bm{\Delta}}_t \right) \\
&\overset{}{=} \text{tr}\left( ({\bm{\Delta}}_t)^\top {\bm{\Delta}}_t \right) = \|{\bm{\Delta}}_t\|_F^2 \ ,
\end{align*}

\noindent Where the last equation uses the orthogonality of $\tilde{\bm{P}}_t$. So we finished the proof.
\end{proof}

\begin{lemma}
\label{lemma2}
Define $\delta = r/n$. For operator $\mathcal{C} = \mathcal{C}_{\text{SVD}}$ defined in Algorithm \hyperref[powersgdalgo]{2}, for any $\bm{\Delta}_t$, we have:
$$\|\mathcal{C}(\bm{\Delta}_t) - \bm{\Delta}_t\|_F^2 \leq (1 - \delta)\|\bm{\Delta}_t\|_F^2 \ .$$
\end{lemma}
\begin{proof}
Let ${\bm{\Delta}}_t = \bm{U}_t \bm{\Sigma}_t \bm{V}_t$ represents its SVD composition, $\tilde{\bm{P}}_t = \bm{U}_t[:, :r]$ is the first $r$ singular vectors. As $m \geq n$, let $\bm{R}_t = \bm{U}_t[:, (r + 1) :]$. It holds that $\bm{I} = \bm{U}_t (\bm{U}_t)^\top = \tilde{\bm{P}}_t (\tilde{\bm{P}}_t)^\top + \bm{R}_t (\bm{R}_t)^\top$. Thus,
\[
\begin{aligned}
\|\mathcal{C}(\bm{\Delta}_t) - \bm{\Delta}_t\|_F^2 &= \left\| \left( \bm{I} - \tilde{\bm{P}}_t (\tilde{\bm{P}}_t)^\top \right) \bm{U}_t \bm{\Sigma}_t (\bm{V}_t)^\top \right\|_{F}^{2} \\
&= \text{tr}\left( \bm{V}_t (\bm{\Sigma}_t)^\top (\bm{U}_t)^\top \left( \bm{I} - \tilde{\bm{P}}_t (\tilde{\bm{P}}_t)^\top \right)^2 \bm{U}_t \bm{\Sigma}_t (\bm{V}_t)^\top \right) \\
&= \text{tr}\left( (\bm{\Sigma}_t)^\top (\bm{U}_t)^\top \bm{R}_t (\bm{R}_t)^\top \bm{U}_t \bm{\Sigma}_t \right),
\end{aligned}
\]

\noindent where the second equation uses $\|\bm{X}\|_{F}^{2} = \text{tr}(\bm{X}^\top \bm{X})$ and the last equation uses $\text{tr}(\bm{A}\bm{B})\\= \text{tr}(\bm{B}\bm{A})$, $(\bm{V}_t)^\top \bm{V}_t = \bm{I}$ and $(\bm{R}_t)^\top \bm{R}_t = \bm{I}$. By $(\bm{R}_t)^\top \tilde{\bm{P}}_t = 0$ and $(\tilde{\bm{P}}_t)^\top \bm{R}_t = 0$
\[
(\bm{U}_t)^\top \bm{R}_t (\bm{R}_t)^\top \bm{U}_t = \begin{pmatrix} (\tilde{\bm{P}}_t)^\top \\ (\bm{R}_t)^\top \end{pmatrix} \bm{R}_t (\bm{R}_t)^\top \begin{pmatrix} \tilde{\bm{P}}_t & \bm{R}_t \end{pmatrix} = \begin{pmatrix} 0_{r \times r} & 0_{r \times (n - r)} \\ 0_{(n - r) \times r} & \bm{I}_{n - r} \end{pmatrix}.
\]

\noindent Let $\sigma_{1} \geq \sigma_{2} \geq \cdots \geq \sigma_{n} \geq 0$ denote the singular values of ${\bm{\Delta}}_t$. This implies
\[
(\bm{\Sigma}_t)^\top (\bm{U}_t)^\top \bm{R}_t (\bm{R}_t)^\top \bm{U}_t \bm{\Sigma}_t = \begin{pmatrix} 0_{r \times r} & 0_{r \times (n - r)} & 0_{r \times (m - n)} \\ 0_{(n - r) \times r} & \text{diag}(\sigma_{r + 1}^2, \cdots, \sigma_{n}^2) & 0_{(n - r) \times (m - n)} \\ 0_{(m - n) \times r} & 0_{(m - n) \times (n - r)} & 0_{(m - n) \times (m - n)} \end{pmatrix}.
\]

\noindent Applying this result yields
\[
\left\| \tilde{\bm{P}}_t (\tilde{\bm{P}}_t)^\top {\bm{\Delta}}_t - {\bm{\Delta}}_t \right\|_{F}^{2} = \text{tr}\left( (\bm{\Sigma}_t)^\top (\bm{U}_t)^\top \bm{R}_t (\bm{R}_t)^\top \bm{U}_t \bm{\Sigma}_t \right) = \sum_{i = r + 1}^{n} \sigma_{i}^{2}  \leq (1 - \delta)\|\bm{\Delta}_t\|_{F}^{2},
\]

\noindent where the inequality uses $\|{\bm{\Delta}}_t\|_{F}^{2} = \text{tr}\left( ({\bm{\Delta}}_t)^\top {\bm{\Delta}}_t \right) = \text{tr}\left( (\bm{\Sigma}_t)^\top \bm{\Sigma}_t \right) = \sum_{i = 1}^{n} \sigma_{i}^{2}$.
\end{proof}
\begin{lemma}
\label{lemma3}
For any \( t \geq 0 \), we have
\[
\mathbb{E} \left\|  \frac{1}{N} \sum_{i=1}^N \bm{e}_{t}^{(i)} \right\|_F^2 \leq \frac{45\tau^2G^2}{\delta^2} \ .
\]
\end{lemma}

\begin{proof}
When \( t = 0 \), the bound trivially holds as \( \bm{e}_{0}^{(i)} = 0 \) for all \( i \).\\
Define \( e_t = \mathbb{E} \left\|  \frac{1}{N} \sum_{i=1}^N \bm{e}_{t}^{(i)} \right\|_F^2 \) with \( e_0 = 0 \).\\
When \( \tau \nmid t \), for all \( \beta > 0 \), we have:
\begin{align*}
e_{t+1} &= \mathbb{E} \left\|  \frac{1}{N} \sum_{i=1}^N \bm{e}_{t+1}^{(i)} \right\|_F^2 \\
&= \mathbb{E} \left\|  \mathcal{C}(\bm{\Delta}_{t}) - \bm{\Delta}_{t} \right\|_F^2 \\
&\leq \mathbb{E} \left[ \|\bm{\Delta}_{t}\|_F^2 \right] \\
&= \mathbb{E} \left\| \frac{1}{N} \sum_{i=1}^N \left( \bm{e}_{t}^{(i)} + \bm{g}_{t}^{(i)} \right) \right\|_F^2 \\
&\leq (1 + \beta)\mathbb{E} \left\|  \frac{1}{N} \sum_{i=1}^N \bm{e}_{t}^{(i)} \right\|_F^2 + (1 + \frac{1}{\beta}) \mathbb{E}\left[\left\|\frac{1}{N}\sum_{i=1}^N   \bm{g}_{t}^{(i)} \right\|_F^2 \right] \\
&\leq (1 + \beta)e_t + (1 + \frac{1}{\beta}) G^2,\tag{11}\label{b1}
\end{align*}
where the first inequality uses Lemma ~\ref{lemma1}, the second inequality uses Young's inequality and Cauchy's inequality, the last inequality uses Assumption \hyperref[assumption]{3}.\\
Similarly, when \( \tau \mid t \), we have:
\begin{align*}
e_{t+1} &= \mathbb{E} \left\|  \frac{1}{N} \sum_{i=1}^N \bm{e}_{t+1}^{(i)} \right\|_F^2 \\
&= \mathbb{E} \left\|  \mathcal{C}(\bm{\Delta}_{t}) - \bm{\Delta}_{t} \right\|_F^2 \\
&\leq (1-\delta) \mathbb{E} \left[ \|\bm{\Delta}_{t}\|_F^2 \right] \\
&\leq (1-\delta)(1 + \beta)e_t + (1-\delta)(1 + \frac{1}{\beta}) G^2 ,
\tag{12}\label{b2}
\end{align*}
where the first inequality uses Lemma \ref{lemma2}, the second inequality use's Young's inequality and Cauchy's inequality, the last inequality uses Assumption \hyperref[assumption]{3}. Apply ~\eqref{b1},\eqref{b2} in a period of $\tau$ steps, we have:
\begin{align*}
e_{n\tau} &\leq (1 + \beta)e_{n\tau-1} + (1 + \frac{1}{\beta}) G^2 \\
&\leq \cdots \leq (1 + \beta)^{\tau-1}e_{(n-1)\tau+1} + (1 + \frac{1}{\beta}) \left(1 + (1+\beta) + \cdots + (1+\beta)^{\tau-2}\right) G^2 \\
&\leq (1-\delta)(1 + \beta)^{\tau}e_{(n-1)\tau} + (1 + \frac{1}{\beta})\tau(1+\beta)^{\tau} G^2.
\tag{13}\label{b3}
\end{align*}
Apply ~\eqref{b3} with \( e_0 = 0 \), we have:
\begin{align*}
e_{n\tau} \leq \frac{(1 + \frac{1}{\beta})\tau(1+\beta)^{\tau}G^2}{1 - (1-\delta)(1+\beta)^\tau} .
\tag{14}\label{b4}
\end{align*}
For any \( t = k\tau + q, \ 1 \leq q \leq \tau \), similarly we have:
\begin{align*}
e_{t} &\leq (1 + \beta)^{q}e_{k\tau} + (1 + \frac{1}{\beta})q(1+\beta)^{q-1} G^2 \\
&\leq (1 + \beta)^{\tau}e_{k\tau} + (1 + \frac{1}{\beta})\tau(1+\beta)^{\tau} G^2 \\
&\leq \frac{(1 + \frac{1}{\beta})\tau(1+\beta)^{\tau}G^2}{1 - (1-\delta)(1+\beta)^\tau} \left( \delta(1+\beta)^\tau + 1 \right) \\
&\leq \frac{2(1 + \frac{1}{\beta})\tau(1+\beta)^{2\tau}G^2}{1 - (1-\delta)(1+\beta)^\tau} , \tag{15}\label{b5}
\end{align*}
\noindent where the first inequality uses ~\eqref{b1} and the third inequality uses ~\eqref{b4}. If we substitute \( \beta = \frac{\delta}{4\tau} \) into ~\eqref{b5}, then \( (1+\beta)^\tau \leq e^{\beta \tau} = e^{\frac{\delta}{4}} \leq 1 + \frac{\delta}{2} \leq \frac{3}{2} \) (as \( \delta \leq 1 \)) and \( 1 + \frac{1}{\beta} = 1 + \frac{4\tau}{\delta} \leq \frac{5\tau}{\delta} \), we get:
\begin{align*}
\mathbb{E} \left\|  \frac{1}{N} \sum_{i=1}^N \bm{e}_{t}^{(i)} \right\|_F^2 
= e_t &\leq \frac{45\tau^2G^2}{2\delta} \cdot \frac{1}{1 - (1-\delta)(1 + \frac{\delta}{2})} 
\leq \frac{45\tau^2G^2}{\delta^2} .
\end{align*}
\end{proof}

\begin{lemma}
\label{lemma4}
For any \( t \geq 0 \), we have
\[
\mathbb{E} \left\|  \bm{m}_t \right\|_F^2 \leq \frac{92\tau^2G^2}{(1-\mu)^2\delta^2} .
\]
\end{lemma}

\begin{proof}
\begin{align*}
\mathbb{E}[\| \mathcal{C}(\bm{\Delta}_t) \|_F^2] &\leq \mathbb{E}[\|\bm{\Delta}_t\|_F^2] \\
&= \mathbb{E} \left\| \frac{1}{N} \sum_{i=1}^N \left( \bm{e}_{t}^{(i)} + \bm{g}_{t}^{(i)} \right) \right\|_F^2 \\
&\leq 2 \left( \mathbb{E} \left\| \frac{1}{N} \sum_{i=1}^N \bm{e}_{t}^{(i)} \right\|_F^2 +  \mathbb{E}\left[\left\|\frac{1}{N}\sum_{i=1}^N   \bm{g}_{t}^{(i)} \right\|_F^2 \right] \right) \\
&\leq 2 \left( \frac{45\tau^2G^2}{\delta^2} + G^2 \right) 
\leq \frac{92\tau^2G^2}{\delta^2} ,\tag{16}\label{b6}
\end{align*}
\noindent where the first inequality uses Lemma ~\ref{lemma1}, the second inequality uses Cauchy's inequality and the last inequality uses Lemma ~\ref{lemma3}. Applying ~\eqref{b6} and use Cauchy's inequality, we have: 
\begin{align*}
\mathbb{E} \left\| \bm{m}_t \right\|_F^2 
&= \mathbb{E} \left[ \left\| \sum_{k=0}^t \mu^{t - k}\mathcal{C}(\bm{\Delta}_k) \right\|_F^2 \right] \\
&\leq \left( \sum_{k=0}^t \mu^{t - k} \right) \left( \sum_{k=0}^t \mu^{t - k} \mathbb{E} \left[ \| \mathcal{C}(\bm{\Delta}_k) \|_F^2 \right] \right) \\
&\leq \left( \sum_{k=0}^t \mu^{t - k} \right)^2 \frac{92\tau^2G^2}{\delta^2} \\
&\leq \frac{92 \tau^2 G^2}{(1 - \mu)^2\delta^2} .
\end{align*}
\end{proof}

\begin{theorem}{(Convergence of PowerSGD+)}
\label{appendixtheorem2}
Under Assumption ~\ref{assumption}--~\ref{ass:uniform}, Define $\delta=r/n$ and $\Delta F=f(\bm{X}_0)-\underset{}{\text{min}}\ f(\bm{X})$, PowerSGD+ with MSGD (Algorithm \hyperref[powersgdalgo]{2}) converges as 
\begin{align*}
&\frac{1}{T}\sum_{t=0}^{T - 1} \mathbb{E} \left[ \|\nabla f(\bm{X}_t)\|_F^2 \right]
\leq\frac{4(1-\mu)}{\eta T}\Delta F +\frac{2L\eta \sigma^2}{(1 - \mu) N}+\frac{548L^2\eta^2\tau^2G^2 }{(1 - \mu)^4\delta^2} .
\end{align*}
\end{theorem}

\begin{proof}
Define:
\begin{align*}
\tilde{\bm{X}}_t &= \bm{X}_t - \frac{\eta}{1 - \mu} \frac{1}{N} \sum_{i=1}^N \bm{e}_{t}^{(i)} ,\\
\bm{Z}_t &= \tilde{\bm{X}}_t - \frac{\eta \mu}{1 - \mu} \bm{m}_{t-1}\ .
\end{align*}
Then we have:
\begin{align*}
\bm{Z}_{t+1}
=& \tilde{\bm{X}}_{t+1} - \frac{\eta \mu}{1 - \mu} \bm{m}_t \\
=& \bm{X}_{t+1} - \frac{\eta}{1 - \mu} \frac{1}{N} \sum_{i=1}^N \bm{e}_{t+1}^{(i)} - \frac{\eta \mu}{1 - \mu} \bm{m}_t \\
=& \bm{X}_t - \frac{\eta}{1 - \mu} \frac{1}{N} \sum_{i=1}^N \bm{e}_{t+1}^{(i)} - \frac{\eta}{1 - \mu} \bm{m}_t \\
=& \bm{X}_t - \frac{\eta}{1 - \mu} \frac{1}{N} \sum_{i=1}^N \bm{e}_{t+1}^{(i)} - \frac{\eta}{1 - \mu} \cdot\\& \left( \mu \bm{m}_{t-1} - \frac{1}{N} \sum_{i=1}^N \bm{e}_{t+1}^{(i)} + \frac{1}{N} \sum_{i=1}^N \bm{e}_{t}^{(i)} + \frac{1}{N} \sum_{i=1}^N \bm{g}_{t}^{(i)} \right) \\
=& \tilde{\bm{X}}_t - \frac{\eta \mu}{1 - \mu} \bm{m}_{t-1} - \frac{\eta}{1 - \mu} \frac{1}{N} \sum_{i=1}^N \bm{g}_{t}^{(i)} \\
=& \bm{Z}_t - \frac{\eta}{1 - \mu} \frac{1}{N} \sum_{i=1}^N \bm{g}_{t}^{(i)} .
\end{align*}

\noindent By the smoothness of the function \( f \), we get:
\begin{align*}
&\mathbb{E}_t[f(\bm{Z}_{t+1})]
\\\leq& f(\bm{Z}_t) + \langle \nabla f(\bm{Z}_t), \mathbb{E}_t[\bm{Z}_{t+1} - \bm{Z}_t] \rangle + \frac{L}{2} \mathbb{E}_t[\|\bm{Z}_{t+1} - \bm{Z}_t\|_F^2] \\
=& f(\bm{Z}_t) - \frac{\eta}{1 - \mu} \left\langle \nabla f(\bm{Z}_t), \mathbb{E}_t \left[ \frac{1}{N} \sum_{i=1}^N \bm{g}_{t}^{(i)} \right] \right\rangle + \frac{L \eta^2}{2(1 - \mu)^2} \mathbb{E}_t \left[ \left\| \frac{1}{N} \sum_{i=1}^N \bm{g}_{t}^{(i)} \right\|_F^2 \right] \\
=& f(\bm{Z}_t) - \frac{\eta}{1 - \mu} \langle \nabla f(\bm{Z}_t), \nabla f(\bm{X}_t) \rangle \\&+ \frac{L \eta^2}{2(1 - \mu)^2} \left[ \|\nabla f(\bm{X}_t)\|_F^2 + \mathbb{E}_t \left[ \left\| \frac{1}{N} \sum_{i=1}^N \bm{g}_{t}^{(i)} - \nabla f(\bm{X}_t) \right\|_F^2 \right] \right] \\
\leq& f(\bm{Z}_t) - \frac{\eta}{1 - \mu} \langle \nabla f(\bm{Z}_t), \nabla f(\bm{X}_t) \rangle + \frac{L \eta^2}{2(1 - \mu)^2} \|\nabla f(\bm{X}_t)\|_F^2 + \frac{L \eta^2 \sigma^2}{2(1 - \mu)^2 N} , \tag{17}\label{b7}
\end{align*}
\noindent where the first inequality uses Assumption ~\ref{assumption} and the last inequality uses Cauchy's inequality and Assumption ~\ref{asp:SGO}. Bound the second term \( -\mathbb{E}\!\left[\langle\nabla f(\boldsymbol{Z}_t), \nabla f(\boldsymbol{X}_t)\rangle\right] \) with Cauchy's inequality, we have:
\begin{align*}
-\mathbb{E}\!\left[\langle\nabla f(\boldsymbol{Z}_t), \nabla f(\boldsymbol{X}_t)\rangle\right] 
&= -\mathbb{E}\!\left[\|\nabla f(\boldsymbol{X}_t)\|_F^2\right] + \mathbb{E}\!\left[\langle\nabla f(\boldsymbol{X}_t) - \nabla f(\boldsymbol{Z}_t), \nabla f(\boldsymbol{X}_t)\rangle\right] \\
&\leq -\frac{1}{2}\mathbb{E}\!\left[\|\nabla f(\boldsymbol{X}_t)\|_F^2\right] + \frac{1}{2}\mathbb{E}\!\left[\|\nabla f(\boldsymbol{X}_t) - \nabla f(\boldsymbol{Z}_t)\|_F^2\right] ,\tag{18} \label{b8}
\end{align*}
And:
\begin{align*}
\mathbb{E}\!\left[\|\nabla f(\boldsymbol{X}_t) - \nabla f(\boldsymbol{Z}_t)\|_F^2\right] 
&\leq L^2 \mathbb{E}\!\left[\|\boldsymbol{X}_t - \boldsymbol{Z}_t\|_F^2\right] \\
&\leq 2L^2 \mathbb{E}\!\left[\|\boldsymbol{X}_t - \tilde{\boldsymbol{X}}_t\|_F^2\right] + 2L^2 \mathbb{E}\!\left[\|\tilde{\boldsymbol{X}}_t - \boldsymbol{Z}_t\|_F^2\right] \\
&= 2L^2\eta^2 \mathbb{E}\!\left[\left\|\frac{1}{N} \sum_{i=1}^N \boldsymbol{e}_{t}^{(i)}\right\|_F^2\right] + \frac{2L^2\eta^2\mu^2}{(1 - \mu)^2} \mathbb{E}\!\left[\|\boldsymbol{m}_{t-1}\|_F^2\right] \\
&\leq \frac{90L^2\eta^2\tau^2G^2}{\delta^2} + \frac{184L^2\eta^2\tau^2\mu^2G^2}{(1 - \mu)^4\delta^2} ,\tag{19} \label{b9}
\end{align*}
\noindent where the last inequality follows Lemma ~\ref{lemma3},~\ref{lemma4}. Applying ~\eqref{b8} to ~\eqref{b7}, ~\eqref{b9} to ~\eqref{b7} and taking expectations, we have:
\begin{align*}
\mathbb{E}\left[f(\bm{Z}_{t+1})\right]
&\leq \mathbb{E}[f(\bm{Z}_t)]
  - \left( \frac{\eta}{2(1 - \mu)} - \frac{L\eta^2}{2(1 - \mu)^2} \right) \mathbb{E} \left[ \|\nabla f(\bm{X}_t)\|_F^2 \right] 
  + \frac{L\eta^2 \sigma^2}{2(1 - \mu)^2 N} \\
  &+ \frac{45L^2\eta^3\tau^2G^2}{(1 - \mu)\delta^2} + \frac{92L^2\eta^3\tau^2\mu^2G^2}{(1 - \mu)^5\delta^2} .\tag{20}\label{b10}
\end{align*}

\noindent Taking total expectation and telescoping ~\eqref{b10} from \(0\) to \(T - 1\), we obtain:

\begin{align*}
&\left( \frac{\eta}{2(1 - \mu)} - \frac{L\eta^2}{2(1 - \mu)^2} \right) \sum_{t=0}^{T - 1} \mathbb{E} \left[ \|\nabla f(\bm{X}_t)\|_F^2 \right]\\
\leq&f(\bm{X}_0) - \mathbb{E}[f(\bm{Z}_T)]+\frac{L\eta^2 \sigma^2 T}{2(1 - \mu)^2 N} +\frac{45L^2\eta^3\tau^2G^2T}{(1-\mu)\delta^2}+\frac{92L^2\eta^3\tau^2\mu^2G^2T }{(1 - \mu)^5\delta^2} .\tag{21}\label{b11}
\end{align*}
\noindent When $\eta \leq \frac{1-\mu}{2L}$,$\frac{\eta}{2(1 - \mu)} - \frac{L\eta^2}{2(1 - \mu)^2} \geq \frac{\eta}{4(1 - \mu)}$, applying ~\eqref{b11} and taking average on \(T\), we get:
\begin{align*}
&\frac{1}{T}\sum_{t=0}^{T - 1} \mathbb{E} \left[ \|\nabla f(\bm{X}_t)\|_F^2 \right]\\
\leq&\frac{4(1-\mu)}{\eta T}(f(\bm{X}_0) - f^*)+\frac{2L\eta \sigma^2}{(1 - \mu) N} +\frac{180L^2\eta^2\tau^2G^2}{\delta^2}+\frac{368L^2\eta^2\tau^2\mu^2G^2 }{(1 - \mu)^4\delta^2}\\
\leq&\frac{4(1-\mu)}{\eta T}\Delta F+\frac{2L\eta \sigma^2}{(1 - \mu) N}+\frac{548L^2\eta^2\tau^2G^2 }{(1 - \mu)^4\delta^2} ,
\end{align*}
\noindent so we finished the proof.
\end{proof}

\begin{corollary}
In Theorem ~\ref{appendixtheorem2}, If we set 
$$\eta=\left(\frac{2L}{1-\mu}+\sqrt{\frac{L\sigma^2T}{2(1-\mu)^2N\Delta F}}+
\sqrt[3]{\frac{274L^2\tau^2G^2T}{(1-\mu)^5\delta^2\Delta F}}\right)^{-1},$$
we get:
\begin{align*}
&\frac{1}{T}\sum_{t=0}^{T - 1} \mathbb{E} \left[ \|\nabla f(\bm{X}_t)\|_F^2 \right]
\leq\sqrt{\frac{32L\sigma^2\Delta F}{NT}} +\frac{8L\Delta F}{T}+\frac{6\sqrt[3]{274}(L\tau G \Delta F)^{2/3} }{T^{2/3}\delta^{2/3}(1-\mu)^{2/3}} .
\end{align*}
\end{corollary}
\section{Experiments details}
In this section, we show the training details of our experiments in Section \hyperref[section5]{5} for reproduction.
\subsection{Fine-tuning tasks on GLUE datasets} We fine-tune pretrained \texttt{RoBERTa-Base} models on GLUE benchmarks for 10 epochs on a 4 × 4090 24GB GPUs cluster with data parallelism at 4. We used AdamW as base optimizer and set compression rank to 4 for all compressors which refer to the original settings in \PowerSGD\  \cite{vogels2019powersgd}. Detailed hyper-parameters are illustrated in Table ~\ref{tab:finetuning_hparams}. In particular, "Restart steps" is only applicable to the \PowerSGD+\ and Galore compressor.  
\begin{table}[htb]
\label{glue-hyperparameters}
\centering
\small
\caption{Hyperparameters of fine-tuning \texttt{RoBERTa-Base} models on GLUE benchmark }
\label{tab:finetuning_hparams}
\begin{tabular}{lcccccccc}
\toprule
 & RTE & CoLA & SST2 & MNLI & QQP & STSB & MRPC & QNLI \\
\midrule
Total Batch Size & 16 & 32 & 32 & 64 & 32 & 32 & 64 & 32 \\
\# Epochs & 10 & 10 & 10 & 10 & 10 & 10 & 10 & 10 \\
Learning Rate & $2\times 10^{-5}$ & $2\times 10^{-5}$ & $2\times 10^{-5}$ & $2\times 10^{-5}$ & $2\times 10^{-5}$ & $2\times 10^{-5}$ & $4\times 10^{-5}$ & $2\times 10^{-5}$ \\
Weight decay & 0.01 & 0.02& 0.01 & 0.02 & 0.01 & 0.02 & 0.02 & 0.01 \\
Restart steps & 50 & 80& 80 & 200 & 200 & 50 & 80 & 100 \\
Rank Config. & \multicolumn{8}{c}{$r=4$} \\
AdamW betas & \multicolumn{8}{c}{(0.9,0.999)} \\
Lr-scheduler & \multicolumn{8}{c}{linear} \\warmup-ratio & \multicolumn{8}{c}{0.06} \\
Max Seq. Len. & \multicolumn{8}{c}{512} \\
\bottomrule
\end{tabular}
\end{table}

\subsection{Pre-training Tasks on C4 Dataset}
We pre-trained LLaMA models on C4 corpus for 10 epochs on a 4 × 4090 24GB GPUs cluster with data parallelism at 4. The 60M and 130M models are trained for 10k and 20k iterations, respectively, with the standard Adam optimizer. We validate across compression rank in $\{4,8\}$ and tuned the hyperparameters refer to \cite{zhao2024galore}. Detailed hyperparameters are illustrated
in Table ~\ref{c4-hyperparameters}. 
\begin{table*}[!htbp]
\centering
\caption{\small Hyperparameters of pre-training LLaMA models on C4 dataset}
\label{c4-hyperparameters}
\setlength{\tabcolsep}{2pt}
\small
\begin{tabular}{lcccccc}
\toprule
& \multicolumn{3}{c}{\textbf{Llama 60M}} & \multicolumn{3}{c}{\textbf{Llama 130M}} \\
\cmidrule(lr){2-4} \cmidrule(lr){5-7}
& Adam & PowerSGD & PowerSGD+ & Adam & PowerSGD & PowerSGD+ \\
\midrule
Training Steps & \multicolumn{3}{c}{10000} & \multicolumn{3}{c}{20000} \\
Warm-up Steps & \multicolumn{3}{c}{1000} & \multicolumn{3}{c}{2000} \\
Maximum Length & \multicolumn{3}{c}{256} & \multicolumn{3}{c}{256} \\
Batch Size & \multicolumn{3}{c}{512} & \multicolumn{3}{c}{512} \\
Batch Size per Device & \multicolumn{3}{c}{128} & \multicolumn{3}{c}{128} \\
Total Training Tokens & \multicolumn{3}{c}{1 310 720 000} & \multicolumn{3}{c}{2 621 440 000} \\
\midrule
Learning Rate & \multicolumn{3}{c}{$\{1\times 10^{-3}, 2\times 10^{-3}, 4\times 10^{-3}\}$} & \multicolumn{3}{c}{$\{1\times 10^{-3}, 2\times 10^{-3}, 4\times 10^{-3}\}$} \\
Warm-up Scheduling & \multicolumn{3}{c}{linear from 0\%} & \multicolumn{3}{c}{linear from 0\%} \\
Learning Rate Decay & \multicolumn{3}{c}{cosine to 0\%} & \multicolumn{3}{c}{cosine to 0\%} \\
Weight Decay & 0.0 & \hspace{5mm}0.0 & 0.0 & 0.0 & \hspace{5mm}0.0 & 0.0 \\
Gradient Clipping & 1.0 & \hspace{5mm}1.0 & 1.0 & 1.0 & \hspace{5mm}1.0 & 1.0 \\
\midrule
Restart steps & - & - & 200 & - & - & 200 \\
\bottomrule
\end{tabular}
\end{table*}

\end{document}